\documentclass[12pt]{article}
\usepackage{pgf,tikz}
\usepackage{calc,graphicx, complexity}
\usepackage{graphics}
\everymath{\displaystyle}
\usepackage{graphicx}
\usepackage{color}
\usepackage{amssymb}
\usepackage{amsmath}
\newtheorem{prethm}{{\bf Theorem}}

\newenvironment{thm}{\begin{prethm}{\hspace{-0.5
				em}{\bf .}}}{\end{prethm}}
\newtheorem{prelemma}{{\bf Lemma}}

\newenvironment{lemma}{\begin{prelemma}{\hspace{-0.5
				em}{\bf .}}}{\end{prelemma}}
\newtheorem{preex}{{\bf Example}}

\newtheorem{preprop}{{\bf Proposition}}

\newenvironment{prop}{\begin{preprop}{\hspace{-0.5em}{\bf .}}}{\end{preprop}}
\newtheorem{precor}{{\bf Corollary}}

\newenvironment{cor}{\begin{precor}{\hspace{-0.5
				em}{\bf .}}}{\end{precor}}
\newtheorem{preremark}{{\bf Remark}}

\newtheorem{preprob}{{\bf Problem}}

\newtheorem{predefin}{{\bf Definition}}

\newenvironment{defin}{\begin{predefin}{\hspace{-0.5
				em}{\bf .}}}{\end{predefin}}
\newtheorem{preconj}{{\bf Conjecture}}

\newtheorem{preprobb}{{\bf Problem}}

\newtheorem{prelem}{{\bf Theorem}}

\newtheorem{precla}{{\bf Claim}}

\newenvironment{proof}{{\bf Proof.}\rm }{\hfill{$\Box$}}

\newtheorem{presolution}{{\bf Solution.}}

\def\newpic#1{}
\def\qed{\ifhmode\unskip\nobreak\fi\quad\ifmmode\Box\else$\Box$\fi}

\title{\vspace{-1.5cm}\Large\bf\noindent Computing welfare and fairness in allocating identical goods with entitlements and general utility functions}
\author{\large\bf Manouchehr Zaker\footnote{mzaker@iasbs.ac.ir}
\vspace{5mm}\\
Department of Mathematics,\\
Institute for Advanced Studies in Basic Sciences,\\
Zanjan 45137-66731, Iran\\
}

\date{}

\begin{document}
\maketitle

\begin{abstract}
\noindent A number of goods are called identical if they provide the same level of utility to each agent. In various real-world instances of fair division scenarios, identical indivisible items are allocated to consumers and demandants with different entitlements. We assume that the utility of $t$ identical items to each agent $A$ equals $f_A(t)$, where $f_A$ is an arbitrary increasing function corresponding to $A$. We present a polynomial time algorithm that determines the maximum weighted Rawlsian and Leximin welfare for scenarios with identical goods and show that the allocation obtained by the algorithm is equitable up to any item (WEQX). Some results concerning restricted utilitarian welfare and existence of WMMS and WEFX allocations are also presented. We introduce a new quantity ``total weighted deficit," for allocations, and by which we obtain a tractable algorithm to achieve equitable allocations for scenarios with identical goods and different weights via compensation by using a minimum number of identical coins. Some result are for scenarios with $k$ goods of different types.
\end{abstract}



\section{Introduction and related literature}

\noindent The theory of fair division/allocation is a rapidly growing and rigorous interdisciplinary framework spanning mathematics, economy, game theory and artificial intelligence, dedicated to the optimal allocation of resources among agents with competing and subjective preferences. The discipline seeks to formalize intuitive and often philosophical concepts of fairness by establishing precise criteria for resource allocation. In addition, some basic issues from social choice theory such as welfare functions have been used in the theory which quantify welfare of each allocation. The theory generally bifurcates into study of divisible goods and indivisible goods viz discrete items that must be allocated whole. We refer to book chapter \cite{BCM} as a reference for fair allocation of indivisible goods and to \cite{CKMPSW} for a recent survey and open problems on this subject. The set of non-negative real numbers and $\{1, \ldots, k\}$ are denoted respectively by $\mathbb{R}^{\geq 0}$ and $[k]$, where $k$ is a positive integer. Let $A=\{A_1, \ldots, A_n\}$ and $G=\{g_1, \ldots, g_m\}$ be sets of agents and indivisible goods, respectively. A utility function $u_i$ corresponding to $A_i$ is a mapping from the power set of $G$ i.e. $2^G$ to $\mathbb{R}^{\geq 0}$. Sets $A$ and $G$ together with utility functions form a division/allocation scenario. If $u_i(T)={\sum}_{g\in T} u_i(g)$ for each subset $T\subseteq G$ then $u_i$ is called additive utility function.

\noindent Given agents $A_1, \ldots, A_n$ with their utility functions, a number of goods say $h_1, \ldots, h_m$ are called identical if for each $i$, $u_i(h_r)=u_i(h_s)$ for each $1\leq r,s \leq m$. In this paper we consider a special case of fair division problem in which $m$ identical goods (in particular $m$ copies of a single indivisible good) are to be allocated to the agents. We say two indivisible goods $g$ and $g'$ are non-identical if there exists an agent $A$ such that $f_A(g)\not= f_A(g')$, where $f_A$ is the utility function of $A$. We say two non-identical items have different types. Suppose that we have $k$ mutually non-identical goods $g_1, \ldots, g_k$ (i.e. $k$ items with different types) and the task is to allocate $m_i$ copies of $g_i$ to the agents, where $m_i$'s are independent of each other. Hence, the goods are called independent since distribution of $g_r$ does not effect distribution of $g_s$. Denote $\frak{m}=(m_1, \ldots, m_k)$. For $i\in [n]$ and $j\in [k]$, $A_i$ has entitlement $w_{ij}$ in the process of allocating $m_j$ identical copies of $g_j$ to the agents. An entitlement $w_{ij}$ in the division process, quantifies the measure of merit, need or demand the agent $A_i$ entitles in the process of allocation of identical copies of $g_j$. In our model we consider general i.e. not necessarily additive utility functions for the agents. For each $i$ and $j$ there exists an increasing function $f_{ij}$ such that a bundle consisting of $x_{ij}$ copies of $g_j$ has utility $f_{ij}(x_{ij})$ for $A_i$. Since we allocate goods and not chores then the utility of $x$ copies of a good $g$ is strictly greater than $y$ copies of $g$ for each agent. Therefore, we assume that utility functions are increasing i.e. $f(x)>f(y)$ whenever $x>y$ and also computable. It follows that the utility functions are invertible. Also $f_{ij}(1)$ is the net value of single good $g_j$ for $A_i$. Throughout we paper we assume that the utility functions are reported truthfully. There are many methods which prevent agents from misreporting. For example, they do not know if allocations are obtained by maximizing utilitarian or Rawlsian welfare, to be defined precisely later. We denote this kind of scenario as $(k;F,W)$-scenario, where $W=[w_{ij}]$ and $F=[f_{ij}]$ are the $n\times k$ matrices of weights and utility functions. A typical allocation corresponding to $\frak{m}$ in $(k;F,W)$-scenario is an $n\times k$ matrix $X=(x_{ij})$ in which $x_{ij}$ copies of $g_j$ are allocated to $A_i$ such that ${\sum}_i x_{ij}=m_j$, for each $j\in [k]$. Notation $k$ means that there are $k$ distinct types in the scenario.


\noindent The scenarios $(k;F,W)$ in this paper in particular $(1;F,W)$-scenarios have several real world instances we briefly explain some of them. In the so-called Political Apportionment the problem is how to distribute a fixed number of seats (identical and indivisible goods) in a legislature among provinces or political parties based on population or political parties vote share (viz agents with different entitlements). Apportionment problem has been widely investigated in the literature e.g. \cite{BY}. Various methods in apportionment theory were explained in \cite{SPDD}. Utilities in common studies of apportionment theory are usually additive but real world setting of the theory needs utility functions satisfying diminishing marginal utility. This postulate states that the first unit of consumption of a good yields more satisfaction or utility than the subsequent units, and there is a continuing reduction in satisfaction or utility for greater amounts. As consumption increases, the additional satisfaction or utility gained from each additional unit consumed falls, a concept known as diminishing marginal utility. The reason why utilities in the instances of apportionment theory satisfy the postulate is that the benefit of gaining one additional seat is highest when a party has 0 seats (gaining representation). Going from 1 to 2 seats is a significant increase in power, whereas going from 100 to 101 seats offers a much smaller marginal increase in relative power or representation. The diminishing marginal utility occurs in the next examples too with a similar argument. Identical goods are also appeared in Cloud Computing Resource Allocation. In distributed systems, a central scheduler must allocate identical computing resources such as identical CPU cores, virtual machine instances, or fixed blocks of RAM to different applications, processes, or client requests or users to maximize overall system performance. Amdahl’s Law \cite{ZLL} asserts that the first few processors allocated to a task significantly speed it up. However, as you add more processors to the same task, the speedup eventually plateaus due to the serial portion of the code. The marginal gain of the 100th CPU core is far less than the marginal gain of the 2nd \cite{ZLL}. Another real world instance of our allocation model is Supply Chain and Inventory Distribution. When a central warehouse has a limited stock of a specific item and needs to distribute it to regional retail stores or distribution centers. The first few units sent to a store are highly valuable because they meet the immediate, high-probability demand. However, sending 500 units to a small rural store yields diminishing returns because the extra units will simply sit in inventory \cite{AG}. As a last but not least example, we mention Vaccine/Supply Allocation in healthcare triage an important issue studied e.g. in \cite{ML}. During a shortage (like a pandemic), health authorities must allocate limited identical supplies to different regions or hospitals possibly with different needs and urgencies. The first batch of vaccines allocated to a region has the highest utility (saving the most vulnerable/high-risk individuals). To study welfare and fairness in these allocation problems we need to investigate $(1;F,W)$-scenarios for suitable $F$ and $W$.

\noindent {\bf The outline of results is as follows.} Rest of this section is devoted to formal definitions and related literature. Corollary \ref{greedy-util} and Theorem \ref{R1item} present polynomial time algorithms to obtain allocations with maximum weighted utilitarian welfare and maximum weighted Rawlsian welfare of $(1;F,W)$-scenarios. Propositions \ref{weqx} and \ref{leximin} provide WEQX+PO and weighted Leximin allocations in polynomial time, respectively. We prove that if utility functions in a $(1;F,W)$-scenario are quasi-multiplicative then the scenario admits a WEFX allocation. To compute allocations with (weighted) maximin share fairness is a polynomial time problem in Proposition \ref{mms}. We introduce ``total weighted deficit" (shortly {\sf twd}) for allocations in general division scenarios. Theorem \ref{polypsi} provides a polynomial time algorithm for computing an allocation maximizing {\sf twd} in a given $(1;F,W)$-scenario. Using the latter invariant, we explain in Proposition \ref{aug} how to obtain $WEQ$ allocations using a minimum number of identical coins offered as compensation. At the end, Theorem \ref{polypsi} is generalized for $(k;F,W)$-scenarios in Proposition \ref{k-polypsi}.


\noindent Let $S$ be a general fair division scenario with $m$ distinct goods $G=\{g_1, \ldots, g_m\}$ and utility functions $u_i$ corresponding to agent $A_i$. An allocation of goods $G$ to the agents corresponds to a partition of $G$ into $n$ bundles $S_1, \ldots, S_n$, so that $S_i$ consists of the items allocated to agent $A_i$. The utility of $S_i$ for $A_i$ is $u_i(S_i)$. A utility function $u_i$ is additive if for each $S\subseteq G$, $u_i(S) ={\sum}_{g \in S} u_i(g)$. There are two main types of fairness criteria widely studied in theories of fair division and resource allocation. The first one is envy-freeness (EF) which is intrapersonal fairness criterion. Existence of EF is not guaranteed in general. There are two main relaxations of EF, one is called EF1 and defined by Budish \cite{Bu} and the other one is EFX studied by Caragiannis et al. in \cite{CKMPSW}. The other fairness notion is an interpersonal criterion equitability (EQ) which has a long history in fairness theory. EQ is not guaranteed to exist in general. Its variants EQ1 and EQX are studied by Freeman et al. in \cite{FSVX}. These notions are defined below. Let $B: \{g_1, \ldots, g_m\}\rightarrow \{A_1, \ldots, A_n\}$ be an allocation for a scenario $S$. Then $B$ offers the bundle $B^{-1}(A_i)$ to $A_i$. We sometimes use $B_i$ instead of $B^{-1}(A_i)$ when the agent $A_i$ is clear. The value of bundle $B^{-1}(A_i)$ for $A_i$ is $u_{A_i}(B^{-1}(A_i))$. By these notation we have the following definition.

\begin{defin}
\noindent Let $B: \{g_1, \ldots, g_m\}\rightarrow \{A_1, \ldots, A_n\}$ be an allocation for a general scenario $S$ with entitlement vector $(w_1, \ldots, w_n)$. Then $B$ is called

\noindent WEF (weighed envy-free) if for each $A_i$ and $A_j$, $u_{A_i}(B_i)/w_i \geq u_{A_i}(B_j)/w_j$,

\noindent WEF1 (weighed envy-free up to one item) if for each $A_i$ and $A_j$ there exists a good $g\in S(A_j)$ with $u_{A_j}(g)\not= 0$ such that $u_{A_i}(B_i)/w_i \geq u_{A_i}(B_j)\setminus \{g\})/w_j)$,

\noindent WEFX (weighted envy-free up to any good) if for each $A_i$ and $A_j$ and for any $g\in B_j$ with $u_{A_j}(g)\not= 0$ one has $u_{A_i}(B_i)/w_i \geq u_{A_i}(B_j)\setminus \{g\})/w_j$).

\noindent WEQ (weighted equitable) if for each $A_i$ and $A_j$, $u_{A_i}(B_i)/w_i \geq u_{A_j}(B_j)/w_j$.

\noindent WEQX (weighted equitable up to any good) if for each agents $A_i$ and $A_j$ and for any $g\in B_j$ with $u_{A_j}(g)\not= 0$ one has $u_{A_i}(B_i)/w_i \geq u_{A_j}(B_j)\setminus \{g\})/w_j$).

\noindent WMMS (weighted maximin share fair allocation) if for each $i$, $u_{A_i}(B_i)\geq WMMS(A_i)$, where $WMMS(A_i)={\max}_{D} {\min}_j (w_i/w_j)u_{A_i}(D_j)$, where the maximum is over all allocations $D$ of $g_1, \ldots, g_m$ to the agents and $D_j=D^{-1}(A_j)$.
\end{defin}

\noindent Social welfare functions are also applied for fair division scenarios. Two major welfare functions are the utilitarian welfare which is based on the Jeremy Bentham's thesis that ``we should aim at the greatest happiness for the greatest number" \cite{B} and Rawlsian (sometimes called egalitarian) welfare which is an implication of John Rawls ``veil of ignorance" and ``difference principle" doctrine that asserts ``social and economic inequalities are to be arranged so that they are to the greatest benefit of the least disadvantaged" \cite{R}.

\begin{defin}
Given an allocation $B: \{g_1, \ldots, g_m\} \rightarrow \{A_1, \ldots, A_n\}$ for a scenario $S$, the weighted utilitarian welfare $WU(B)$ of $B$ equals to ${\sum}_{i=1}^n w_iu_i(B^{-1}(A_i))$. The weighted Rawlsian welfare $WR(B)$ of $B$ equals ${\min}_{i} u_i(B^{-1}(A_i))/w_i$. An allocation $B$ is utilitarian (resp. Rawlsian or maximin) if $WU(B)={\max}_{B'} WU(B')$ (resp. $WR(B)={\max}_{B'} WR(B')$), where the maximum is taken over all allocations $B'$.
\end{defin}

\noindent An allocation $B$ consisting of the bundles $B_1, \ldots, B_n$ is Pareto optimal (shortly PO) if and only if there is no other allocation $B'$ such that $u_i(B'_i) \geq u_i(B_i)$ for each $i$ and $u_j(B'_j) > u_j(B_j)$ for some $j$. Leximin is a social welfare ordering which is obtained by completion of Rawlsian doctrine. A maximin allocation is not necessarily PO but any Leximin allocation is PO. We need to define a weighted version of Leximin. Let $\beta=(b_i)_{i=1}^n$ and $\beta'=(b'_i)_{i=1}^n$ be two non-decreasing sequences of real number. We say $\beta'$ is lexicographically improvement of $\beta$ (denoted by $\beta \prec \beta'$) if there exists some $k\geq 0$ such that $b_i=b'_i$ for each $i\in [k]$ and $b_{k+1}< b'_{k+1}$.

\begin{defin}
Given a scenario $S$ with entitlement vector $(w_1, \ldots, w_n)$, let $B$ and $B'$ be two allocations for $S$. Denote for simplicity $B_i=B^{-1}(A_i)$ and $B'_i=(B')^{-1}(A_i)$ for each agent $A_i$. Sort non-decreasingly the elements of $\{u_{A_i}(B_i)/w_i: i=1, \ldots, n\}$ as $u_{A_{i_1}}(B_{i_1})/w_{i_1} \leq u_{A_{i_2}}(B_{i_2})/w_{i_2} \leq \cdots \leq u_{A_{i_n}}(B_{i_n})/w_{i_n}$, and the elements of $\{u_{A_i}(B'_i)/w_i: i=1, \ldots, n\}$ as $u_{A_{j_1}}(B'_{j_1})/w_{j_1} \leq \cdots \leq u_{A_{j_n}}(B'_{j_n})/w_{j_n}$. We say $B'$ is Leximin improvement of $B$ if $$\big[\frac{u_{A_{i_1}}(B_{i_1})}{w_{i_1}}, \ldots, \frac{u_{A_{i_n}}(B_{i_n})}{w_{i_n}}\big]\prec
\big[\frac{u_{A_{j_1}}(B'_{j_1})}{w_{j_1}}, \ldots, \frac{u_{A_{j_n}}(B'_{j_n})}{w_{j_n}}\big].$$
\noindent Allocation $B$ is called weighted Leximin allocation if no allocation is Leximin improvement of $B$.
\end{defin}

\noindent A practical importance of equitability has been highlighted in an experimental study conducted by Herreiner and Puppe \cite{HP}. They asked human subjects to deliberate over an assignment of indivisible goods subject to a time limit. It was found that the chosen outcomes were equitable. The effect of the envy-freeness criterion is limited to situations in which other fairness criteria are not applicable. Sun et al. in \cite{SCD} provide a complete picture of the computational complexity of (i) deciding the existence of an EQX/EQ1 and welfare-maximizing allocation; (ii) computing a welfare maximizer among all EQX/EQ1 allocations. For the trade-off between fairness and efficiency, they investigate efficiency loss under these fairness constraints and establish the price of fairness. Computational complexity of determining maximum welfare is a hot research area and widely studied in fairness theory. Aziz et al. in \cite{AHMS} investigate the run time complexity of computing allocations that are both fair and maximize the utilitarian social welfare. Aziz et al. focus on two computational problems: (1) Among the utilitarian-maximal allocations, decide whether there exists one that is also fair; (2) among the fair allocations, compute one that maximizes the utilitarian welfare. They show that both problems are strongly $\NP$-hard when the number of agents is variable, and remain $\NP$-hard for a fixed number of agents greater than two. In \cite{G}, some hardness results and polynomial time
approximation algorithms for several variants of maximin welfare problem are presented. In \cite{FSVX}, the authors study equitable allocations of indivisible goods among agents with additive valuations. They consider equitability in conjunction with other well-studied notions of fairness and economic efficiency and show that the Leximin algorithm produces an allocation that satisfies equitability up to any good and Pareto optimality. They also give a novel algorithm that guarantees Pareto optimality and equitability up to one good in pseudopolynomial time. Maximin share fairness was introduced in \cite{Bu} and its weighted version was investigated by Farhadi et al. in \cite{FGHLPSSY}. Rawlsian welfare has practical applications. For example \cite{JWCH} studies Rawlsian principle of fairness in push and pull supply chains and proves its impact on chain performance and firms’ profits. They show that the Rawlsian principle adopted by individual firms can achieve not only fairness but also Pareto efficiency. Gorantla et al. \cite{GMV} study the problem of envy-free allocating a multiset $X=\{{g_a}^{x_a}: a\in [t]\}$ of indivisible items with additive valuations, where $x_a>0$ denotes the multiplicity of $g_a$ in the scenario. Igarashi et al. in \cite{ILNN} initiated the study of the repeated fair division of indivisible goods and chores, and propose a formal model for this scenario. They showed that, if the number of repetitions is a multiple of the number of agents, there always exists a sequence of allocations that is proportional and Pareto-optimal. On the other hand, irrespective of the number of repetitions, an envy-free and Pareto-optimal sequence of allocations may not exist. They also showed that if the number of repetitions can be chosen freely, then envy-free and Pareto-optimal allocations are achievable for any number of agents.

\section{Utilitarian and restricted utilitarian welfare}

\noindent Suppose we have a $(k,F,W)$-scenario and a task is to allocate $m_j$ identical items of type $g_j$ (which form a set say $G_j$) to the agents for each $j\in [k]$ and optimize a welfare function. In general, even in case $k=1$, it is not known that we can allocate the goods item to item and achieve optimal allocation. If this happens then it is a case in which the sequential allocation is optimal. We consider general but computable utility functions. Then for each $j\in [k]$ and $i\in [n]$ there exists a computable and strictly increasing function $f_{ij}:\mathbb{R}^{\geq 0} \rightarrow \mathbb{R}^{\geq 0}$ such that utility of $y$ items of type $g_j$ for $A_i$ is $f_{ij}(y)$. Call a function $f:\mathbb{R}^{\geq 0} \rightarrow \mathbb{R}^{\geq 0}$ concave if $f(x+1)-f(x)$ is a non-increasing function in $x$. Note that if a utility function is concave then it is diminishing marginal utility. The weighted utilitarian welfare of an allocation $X=(x_{ij})$ for a $(k,F,W)$-scenario $S$ is defined as $W_{\mathfrak{m}}U(S,X)={\sum}_{i=1}^n \big[{\sum}_{j=1}^k w_{ij}f_{ij}(x_{ij})\big]$ and the maximum utilitarian welfare of $S$ is defined as $M_{\mathfrak{m}}WU(S)={\max} W_{\mathfrak{m}}U(S,X)$, where $\mathfrak{m}=(m_1, \ldots, m_k)$ and the maximum is over all allocations $X$ which allocates $m_j$ copies of $g_j$ to the agents.
 $$M_{\mathfrak{m}}WU(S)=\max_{X=(x_{ij}):\forall j\in [k],\sum_{i=1}^n x_{ij} = m_j} W_{\mathfrak{m}}U(S,X).$$

\noindent We sometimes denote $M_{\mathfrak{m}}WU(S)$ by $M_mWU(S)$ when we insist on the total number of items i.e. $m$ instead of its summands $m_1, \ldots, m_k$. In the following we present a sequential algorithm which offers each next item to the agent which maximizes $w_{ij}(f_{ij}(x_{ij}+1)-f_{ij}(x_{ij}))$ over all pairs $(i,j)\in [n]\times [k]$, where $x_{ij}$ is the utility of $A_i$ at the present step. Note that if in a $(1;F,W)$-scenario, one of $f_i$'s say $f_k$ is convex (i.e. $f_k(x+1)-f_k(x)$ is non-decreasing) and at some step an item $g_j$ has been offered to $A_k$ then all next items should be allocated to $A_k$. It follows that the non-trivial case of computing $M_mWU$ is when each $f_i$ is concave. This assertion holds for $(k,F,W)$-scenarios too. Recall that for an allocation $X=[x_{ij}]$, utility of each agent $A_i$ equals to $u_i(X)={\sum}_j f_{ij}(x_{ij})$.

\begin{prop}
In a $(k;F,W)$-scenario $S$ with $n$ agents, let $W=[w_{ij}]$ and $F=[f_{ij}]$ be the matrices of weights and utility functions, where each $f_{ij}$ is concave. Let $\mathfrak{m}=(m_1, \ldots, m_k)$ and the task is to allocate $m_j$ identical goods of type $g_j$, $j\in [k]$. Let $m={\sum}_j m_j$. Let $X^{m-1}=[x^{m-1}_{ij}]$ be an allocation of $m-1$ goods such that $M_{m-1}WU={\sum}_i u_i(X^{m-1})$. Let also $\Delta_{ij}^{m-1}= w_{ij}\big(f_{ij}(x^{m-1}_{ij}+1)-f_{ij}(x^{m-1}_{ij})\big)$. Then $$M_mWU=M_{m-1}WU+ \max_{(i,j):\sum_{r=1}^n x^{m-1}_{rj}<m_j} \Delta_{ij}^{m-1}.$$\label{recursive-util}
\end{prop}

\noindent \begin{proof}
Write for simplicity ${\max}_{(i,j):\sum_{r=1}^n x^{m-1}_{ij}<m_j} \Delta_{ij}^{m-1}=\widetilde{\max}_{ij} \Delta_{ij}^{m-1}$. We first prove that $M_mWU\leq M_{m-1}WU+ \widetilde{\max}_{ij} \Delta_{ij}^{m-1}$. For each $j\in [k]$, denote the $m_j$ goods identical to $g_j$ as $g_j^1, \ldots, g_j^{m_j}$, where $\sum m_j =m$. Set $G_j=\{g_j^1, \ldots, g_j^{m_j}\}$. Let
$$Y:G_1 \cup \cdots \cup G_k \rightarrow \{A_1, \ldots, A_n\}$$
\noindent be an allocation of $m_j$ copies of type $g_j$, $j\in [k]$ and overall $m$ goods to the agents $A_i$, $i\in [n]$ such that $Y$ offers $y_{ij}$ goods of type $g_j$ to $A_i$, $i\in [n]$ and ${\sum}_{ij} w_{ij}f_{ij}(y_{ij})=M_mWU$.
Consider an allocation of $m-1$ goods in which $A_i$ receives $x_{ij}^{m-1}$ goods of type $g_j$ and $M_{m-1}WU={\sum}_{ij} w_{ij}f_{ij}(x^{m-1}_{ij})$. Assume that $p\in [n]$ and $q\in [k]$ are such that
$$w_{pq}\big(f_{pq}(x_{pq}^{m-1}+1)-f_{pq}(x_{pq}^{m-1})\big)=\max_{ij} \Delta_{ij}^{m-1}= \max_{ij} w_{ij}\big(f_{ij}(x_{ij}^{m-1}+1)-f_{ij}(x_{ij}^{m-1})\big).$$
\noindent Since $X^{m-1}$ allocates less than $m$ goods then for some $s\in [k]$, $X^{m-1}$ allocates less than $m_s$ goods from $G_j$. It follows that ${\sum}_i x^{m-1}_{is} < m_s$. We claim that there exists $r\in [n]$ such that $x_{rs}^{m-1}\leq y_{rs}-1$. Otherwise, if $y_{is}\leq x_{is}^{m-1}$ for each $i\in [n]$ then $m_s={\sum}_{i} y_{is} \leq {\sum}_{i} x_{is}^{m-1}<m_s$, a contradiction. Now, since $x_{rs}^{m-1}\leq y_{rs}-1$ and $f_{rs}$ is concave then we have the following
$$w_{rs}(f_{rs}(y_{rs})-f_{rs}(y_{rs}-1))\leq w_{rs}(f_{rs}(x_{rs}^{m-1}+1)-f_{rs}(x_{rs}^{m-1}))$$ $$\leq w_{pq}\big(f_{pq}(x_{pq}^{m-1}+1)-f_{pq}(x_{pq}^{m-1})\big)= \widetilde{\max}_{ij} \Delta_{ij}^{m-1}.~~~~~~~~~\clubsuit$$
\noindent Note that in the last inequality, since ${\sum}_i x^{m-1}_{is} < m_s$ then by the choice of $(p,q)$, $f_{rs}(x_{rs}^{m-1}+1)-f_{rs}(x_{rs}^{m-1})\leq f_{pq}(x_{pq}^{m-1}+1)-f_{pq}(x_{pq}^{m-1})$.

\noindent Since $x_{rs}^{m-1}\leq y_{rs}-1$ then $y_{rs}\not=0$. Hence the bundle of $A_r$ has an item of type $g_s$ under the allocation $Y$. Let $g^x_s$ be a good of type $g_s$ in the bundle of $A_r$ under the allocation $Y$.
Let $Y_0$ be an allocation obtained by restriction of $Y$ on $(G_1\cup \cdots \cup G_k) \setminus \{g^x_s\}$. Note that $Y$ offers $g^x_s$ to $A_r$. Then for each $(i,j)\not= (r,s)$, $Y^{-1}(A_i)\cap G_j=Y_0^{-1}(A_i)\cap G_j$ and $Y^{-1}(A_r)\cap G_s=(Y_0^{-1}(A_r)\cap G_s)\cup \{g^x_s\}$. It follows that $y_i=|B^{-1}(A_i)|=|B_0^{-1}(A_i)|$ for $i\not=j$ and $y_j=|B^{-1}(A_j)|=|B_0^{-1}(A_j)|+1$. We obtain the following expansion.
$$M_mWU=\sum_{ij} w_{ij}f_{ij}(y_{ij})=\sum_{ij} w_{ij}f_{ij}(|Y^{-1}(A_i)\cap G_j|)$$
$$=\sum_{(i,j)\not=(r,s)} w_{ij}f_{ij}(|Y_0^{-1}(A_i)\cap G_j|) + w_{rs}f_{rs}(|Y_0^{-1}(A_r)\cap G_s|+1)$$
$$={\sum}_{ij} w_{ij} f_{ij}(|Y_0^{-1}(A_i)\cap G_j|)+ w_{rs}\big[f_{rs}(|Y_0^{-1}(A_r)\cap G_s|+1)-f_{rs}(|Y_0^{-1}(A_r)\cap G_s|)\big].$$
\noindent The first summation is the welfare of an allocation with $m-1$ total goods and hence is at most $M_{m-1}WU$. For the second term note that $|Y_0^{-1}(A_r)\cap G_s|=|Y^{-1}(A_r)\cap G_s|-1$. It follows that
$$\big[f_{rs}(|Y_0^{-1}(A_r)\cap G_s|+1)-f_{rs}(|Y_0^{-1}(A_r)\cap G_s|)\big]=f_{rs}(y_{rs})-f_{rs}(y_{rs}-1).$$
\noindent We obtain the following in which we use $\clubsuit$. Proof of the desired inequality is implied.
$$M_mWU \leq M_{m-1}WU +w_{rs}\big(f_{rs}(y_{rs})-f_{rs}(y_{rs}-1)\big)\leq M_{m-1}WU + \widetilde{\max}_{ij} \Delta_{ij}^{m-1}.$$
\noindent The converse inequality $M_mWU\geq M_{m-1}WU+ \widetilde{\max}_{ij} \Delta_{ij}^{m-1}$ is proved simply. Recall the allocation $X^{m-1}=[x_{ij}^{m-1}]$ such that $M_{m-1}WU={\sum}_{ij} w_{ij}f_{ij}(x^{m-1}_{ij})$. And also
$p\in [n]$ and $q\in [k]$ such that
$$w_{pq}\big(f_{pq}(x_{pq}^{m-1}+1)-f_{pq}(x_{pq}^{m-1})\big)=\widetilde{\max}_{ij} \Delta_{ij}^{m-1}.$$
\noindent Now, to allocate $m$ goods we first allocate $m-1$ goods in such a way that $A_i$ receives $x_{ij}^{m-1}$ goods of type $g_j$, for each $i\in [n]$ and $j\in [k]$. And then offer a last item of type $g_q$ to $A_p$. The resulting allocation has utilitarian welfare $M_{m-1}WU+ \widetilde{\max}_{ij} \Delta_{ij}^{m-1}$. Therefore $M_mWU\geq M_{m-1}WU+ \widetilde{\max}_{ij} \Delta_{ij}^{m-1}$, as desired.
\end{proof}

\noindent Proposition \ref{recursive-util} suggests the following greedy algorithm with $m+1$ steps for determining $M_mWU(S)$. At the first step set $x_{ij}^0=0$, for each $i,j$. Obviously $M_0U=0$. Let $(r,s)$ be such that $f_{rs}(x_{rs}^0+1)-f_{rs}(x_{rs}^0))$ has maximum value. Define $x_{rs}^1=x_{rs}^0+1$ and $x_{ij}^1=x_{ij}^0$, for all $(i,j)\not=(r,s)$. It is clear that $M_1U=w_{rs}f_{rs}(x_{rs}^0+1)-f_{rs}(x_{rs}^0))$. In general, let $t$ be an integer with $2\leq t \leq m-1$ and let for each $i$ and $j$, $x_{ij}^t$ be the number of items in $G_j$ offered to $A_i$ by the algorithm at step $t$. Let $(p,q)$ be a pair subject to ${\sum}_i x_{iq}^t< m_q$ which maximizes  $w_{ij}(f_{ij}(x_{ij}^t+1)-f_{ij}(x_{ij}^t))$. Define $x_{pq}^{t+1}=x_{pq}^t+1$ (i.e. an item from $G_q$ is offered to $A_p$) and $x_{ij}^{t+1}=x_{ij}^t$, for $(i,j)\not=(p,q)$. By Proposition \ref{recursive-util} $M_{t+1}U=M_tU + w_{pq}(f_{pq}(x_{pq}^p+1)-f(x_{pq}^p))$. At step $t=m$ the algorithm obtains $M_mWU(S)$. At each step there are $\mathcal{O}(nk)$ arithmetic and comparison operations. Then the total complexity is $\mathcal{O}(mnk)$.

\noindent {\bf Name:} GUA (Greedy Utilitarian Allocation)
\newline {\bf Input:} $(k;F,W)$-scenario $S$, $F=[f_{ij}]$, $W=[w_{ij}]$, $\mathfrak{m}=(m_1, \ldots, m_k)$, $m={\sum}_j m_j$
\newline {\bf Output:} An allocation $X=[x_{ij}]$ maximizing ${\sum}_{ij} w_{ij}f_{ij}(x_{ij})$ such that ${\sum}_i x_{ij}=m_j$ (i.e. $M_mWU(S))$

\noindent 1. $\forall (i,j)\in [n]\times [k], x_{ij}^0=0$, $t=0$\\
\noindent 2. {\bf while} $t<m$ {\bf do}\\
\noindent 3. ~~~~~ $(r,s)= \arg {\max}_{(i,j):~{\sum}_i x_{ij}^t<m_j} [w_{ij}(f_{ij}(x_{ij}^t+1)-f_{ij}(x_{ij}^t))]$\\
\noindent 4. ~~~~~ $x_{rs}^{t+1}=x_{rs}^t+1$ $\&$ $\forall (i,j)\not= (r,s)$, $x_{ij}^{t+1}=x_{ij}^t$\\
\noindent 5. ~~~~~ $t \leftarrow t+1$\\
\noindent 6. {\bf end while}\\
\noindent 7. {\bf return} $x_{ij}^m$, $(i,j)\in [n] \times [k]$\\

\begin{cor}
Let $S$ be a $(k;F,W)$-scenario with $n$ agents and $m={\sum}_j m_j$ items, where each utility function in $F$ is concave. The task is to allocate $m_j$ identical goods of type $g_j$, for each $j\in [k]$. Then GUA of time complexity $\mathcal{O}(mnk)$ outputs an allocation $x_{ij}^m$ such that ${\sum}_{ij} w_{ij}f_{ij}(x^m_{ij})=M_mWU(S)$.
\end{cor}\label{greedy-util}

\noindent If we apply Proposition \ref{recursive-util} for utility function $f_i(x)=\ln (x)$ then $\max {\sum}_{i} w_{i}f_{i}(x_i)=\ln \max {\prod}_i x_i^{w_i}$. The latter value is known as weighted maximum Nash welfare of $S(1;F,W)$. We obtain

\begin{cor}
The weighted maximum Nash welfare for a $(1;F,W)$-scenario $S$ is obtained by a polynomial time algorithm.
\end{cor}

\noindent In a restricted version of maximum utilitarian welfare problem, for each $i\in [n]$ an upper bound $\ell_i$ is imposed for total utility of agent $A_i$ and the task is to allocate $m$ identical items in a given scenario $S(1;F,W)$. Let also $\ell=(\ell_1, \ldots, \ell_n)$ be a vector with non-negative entries. First define the class of allocations $X$ consistent with $\ell$ as $\mathcal{C}(S,\ell,m)=\big\{X=(x_i):\forall r \in [n], {\sum}_{i=1}^n x_{i}=m~\&~f_r(x_r)\leq \ell_r\big\}$.
\noindent The maximum restricted utilitarian welfare $M_mU^{\leq \ell}(S)$ is defined as
\begin{equation}
M_mU^{\leq \ell}(S)=\begin{cases}
\max_{X=(x_i): X\in \mathcal{C}} {\sum}_{i=1}^n w_if_i(x_i) \quad\hspace{0cm} if~\mathcal{C}(S,\ell,m)\not= \varnothing\\
			
-\infty \quad\hspace{4.3cm} otherwise.
\end{cases}
\end{equation}

\noindent Note that it is possible that for some $m$, $\mathcal{C}(S,\ell,m)=\varnothing$. In this case $M_mU^{\leq \ell}(S)=-\infty$. Let $X^{m-1}$ be an allocation of $m-1$ identical goods. Recall that $\Delta_i^{m-1}= w_i\big(f_i(x^{m-1}_i+1)-f_i(x^{m-1}_i)\big)$. Define the set of indexes $r$ which are extendible from $X^{m-1}$ and keeping the extended allocation $\ell$-consistent as $I(X^{m-1},\ell)=\big\{r\in [n]: f_{r}(x^{m-1}_{r}+1)\leq \ell_r\}$. Also define

\begin{equation}
\tau(X^{m-1},\ell)=\begin{cases}
\max_{i\in I} \Delta_i^{m-1} \quad\hspace{0cm} if~I(X^{m-1},\ell)\not= \varnothing\\
			
-\infty \quad\hspace{1.8cm} otherwise.
\end{cases}
\end{equation}

\noindent The following proposition is counterpart of Proposition \ref{recursive-util}.

\begin{prop}
Let $S$ be $(1;F,W)$-scenario with $n$ agents, where $W=(w_i)$ is a weight vector and $F=(f_i)$ such that each $f_i$ is concave function. Let also $\ell=(\ell_1, \ldots, \ell_n)$ be a vector of bounds. Suppose a task is to allocate $m$ identical goods such that the total utility received by $A_i$ is at most $\ell_i$, for each $i$. Let $X^{m-1}=(x^{m-1}_i)_i$ be an allocation of $m-1$ goods such that $M_{m-1}U^{\leq \ell}(S)={\sum}_i u_i(X^{m-1})$. Then $$M_mU^{\leq \ell}(S)=M_{m-1}U^{\leq \ell}(S)+ \tau(X^{m-1},\ell).$$
\end{prop}\label{restrict-util}

\noindent \begin{proof}
\noindent First note that $M_mU^{\leq \ell}(S)=-\infty$ if and only if either no allocation of $m-1$ goods is $\ell$-consistent or no $\ell$-consistent allocation of $m-1$ goods is extendible to an $\ell$-consistent allocation of $m$ goods if and only if either $M_{m-1}U^{\leq \ell}(S)=-\infty$ or $\tau(Y^{m-1},\ell)=-\infty$ for every allocation $Y^{m-1}$ with $m-1$ goods for $S$. Then in this case the asserted equality holds. Assume hereafter that $M_mU^{\leq \ell}(S)\not=-\infty$. It follows that $M_{m-1}U^{\leq \ell}(S)\not=-\infty$. Take an allocation $X^{m-1}=(x^{m-1}_i)$ with $m-1$ items such that $M_{m-1}U^{\leq \ell}(S)={\sum}_i u_i(X^{m-1})$. We first prove that $I(X^{m-1},\ell)\not= \varnothing$ and hence $\tau(X^{m-1},\ell)\not= -\infty$. Next, we prove that $M_mU^{\leq \ell}(S) \leq M_{m-1}U^{\leq \ell}(S)+ \tau(X^{m-1},\ell)$.

\noindent Let $Y:\{g_1, \ldots, g_m\} \rightarrow \{A_1, \ldots, A_n\}$ be an $\ell$-consistent allocation of $m$ goods to the agents which offers $y_i$ goods to $A_i$, $i\in [n]$ and ${\sum}_i w_if_i(y_i)=M_mU^{\leq \ell}(S)$. Assume that $p\in [n]$ is such that
$$w_p\big(f_p(x_p^{m-1}+1)-f_p(x_p^{m-1})\big)=\tau(X^{m-1},\ell)= \max_{i\in I} w_i\big(f_i(x_p^{m-1}+1)-f_i(x_i^{m-1})\big).$$
\noindent We claim that there exists $r\in [n]$ such that $x_r^{m-1}\leq y_r-1$. Otherwise, if $y_i\leq x_i^{m-1}$ for each $i\in [n]$ then $m={\sum}_{i} y_i \leq {\sum}_{i} x_{i}^{m-1}=m-1$, a contradiction. On the other hand, $r\in I(X^{m-1},\ell)$ since otherwise $f_r(x_r^{m-1}+1)>\ell_r$. Now, $y_r \geq x_r^{m-1}+1$ implies that $f_r(y_r)\geq f_r(x_r^{m-1}+1) >\ell_r$. This contradicts the fact that $Y$ is $\ell$-consistent. It follows that $r\in I(X^{m-1},\ell)$ and $\tau(X^{m-1},\ell)\not= -\infty$. Now, since $x_{r}^{m-1}\leq y_{r}-1$, $r\in I$ and $f_{r}$ is concave we obtain the following inequalities. The first inequality is obtained by the concavity of $f_r$ and the second is obtained by $r\in I$ and maximal property of $p$.
$$w_r(f_r(y_r)-f_r(y_r-1))\leq w_r(f_r(x_r^{m-1}+1)-f_r(x_r^{m-1}))$$ $$\leq w_p\big(f_p(x_p^{m-1}+1)-f_p(x_p^{m-1})\big)=\tau(X^{m-1},\ell).~~~~~~~~~\clubsuit$$
\noindent The rest of proof is similar to the proof of Proposition \ref{recursive-util}. Let $Y_0$ be as in the proof.
$$M_mU^{\leq \ell}(S)={\sum}_i w_if_i(y_i)=\sum_i w_if_i(|Y^{-1}(A_i)|)$$
$$=\sum_{i\not=r} w_if_i(|Y_0^{-1}(A_i)|) + w_rf_r(|Y_0^{-1}(A_r)|+1)$$
$$={\sum}_i w_if_i(|Y_0^{-1}(A_i)|)+ w_r\big[f_r(|Y_0^{-1}(A_r)|+1)-f_r(|Y_0^{-1}(A_r)|)\big].$$
\noindent The first summation is the welfare of an $\ell$-consistent allocation with $m-1$ total goods and hence is at most $M_{m-1}U^{\leq \ell}$. For the second term note that $|Y_0^{-1}(A_r)|=|Y^{-1}(A_r)|-1$. It follows that
$$\big[f_r(|Y_0^{-1}(A_r)|+1)-f_r(|Y_0^{-1}(A_r)|)\big]=f_r(y_r)-f_r(y_r-1).$$
\noindent We obtain the following in which we use $\clubsuit$. Proof of the desired inequality is deduced.
$$M_mU^{\leq \ell}(S) \leq M_{m-1}U^{\leq \ell}(S) +w_r\big(f_r(y_r)-f_r(y_r-1)\big)\leq M_{m-1}U^{\leq \ell}(S) + \tau(X^{m-1},\ell).$$
\noindent The converse inequality is clear.
\end{proof}

\noindent The following algorithm obtained by Proposition \ref{restrict-util} is a modification of Greedy Utilitarian Allocation and determines $M_mU^{\leq \ell}(S)$.

\noindent {\bf Name:} GRUA (Greedy Restricted Utilitarian Allocation)
\newline {\bf Input:} $S(1;F,W)$, $F=(f_i)$, $W=(w_i
)$, integer $m$, upper bounds $(\ell_1, \ldots, \ell_n)$
\newline {\bf Output:} An allocation $X=(x_i)$ maximizing ${\sum}_i w_if_i(x_i)$ over all $(x_i)$ such that ${\sum}_i x_i=m$ and $\forall i, f_i(x_i)\leq \ell_i$

\noindent 1. $t=0$, $x_i^0=0$ for each $i\in [n]$\\
\noindent 2. {\bf while} $t<m$ {\bf do}\\
\noindent 3. ~~ {\bf if} $\exists i$ such that $f_i(x_i^t+1)\leq \ell_i$ {\bf then}\\
\noindent 4. ~~~ $p= \arg \max \{w_i(f_i(x_i^t+1)-f_i(x_i^t)): f_i(x_i^t+1)\leq \ell_i\}$\\
\noindent 5. ~~~~~~~ $x_p^{t+1}=x_p^t+1$ $\&$ $x_i^{t+1}=x_i^t, \forall i\not= p$\\
\noindent 6. ~~~~~~~~~ $t \leftarrow t+1$\\
\noindent 7. ~~~~~~~ {\bf else}\\
\noindent 8. ~~~~~~{\bf return} $M_mU^{\leq \ell}(S)=-\infty$\\
\noindent 9. {\bf end while}\\
\noindent 10. {\bf return} $x_i^m$, $i\in [n]$\\

\noindent Proof of the following result is similar to that of Proposition \ref{recursive-util}. We omit the details.

\begin{prop}
Given a $(1;F,W)$-scenario $S$ and a bound $\ell_i$ for each agent $A_i$, GRUA of time complexity $\mathcal{O}(nm)$ determines $M_mU^{\leq \ell}(S)$ and in case that $M_mU^{\leq \ell}(S)\not= -\infty$ it provides an allocation $(x_1, \ldots, x_n)$ satisfying $M_mU^{\leq \ell}(S)= {\sum}_i w_if_i(x_i)$.\label{r-utilitar}
\end{prop}

\noindent We end this section by a result concerning Pareto optimality. Let $X=[x_{ij}]$ be an allocation for a $(k;F,W)$-scenario which allocates $m_j$ identical copies of type $g_j$ for each $j\in [k]$. Let $Y=[y_{ij}]$ be a Pareto improvement of $X$. Then there exists $(r,s)\in [n]\times [k]$ such that $f_{rs}(y_{rs})>f_{rs}(x_{rs})$ and hence $y_{rs}>x_{rs}$. We have $m_s={\sum}_i y_{is}> {\sum}_i x_{is}=m_s$. This contradiction proves the following.

\begin{prop}
Every allocation for a $(k;F,W)$-scenario is Pareto Optimal.\label{pareto}
\end{prop}

\section{Rawlsian and Leximinal welfare and more}

\noindent Recall that the utility functions are computable and strictly increasing and then invertible. In this section we assume that $f^{-1}$ is computable for each utility function $f$. This means that for each $y$ in the range of $f$, the equation $f(x)=y$ can be solved for $x$ (or a suitable rational approximation of $x$) by a deterministic algorithm. Hence, we do not concern the complexity of computing $x$ and only consider the cost of algorithms in terms of scenario parameters. By a weighted maximin allocation in a $(1; F, W)$-scenario $S$ we mean an allocation $(x_1, \ldots, x_n)$ of $m$ items such that ${\min}_i f_i(x_i)/w_i$ is the maximin welfare of the scenario.

\begin{prop}
Let $S$ be a $(1; F, W)$-scenario with $m$ copies of a single $\underline{\it divisible}$ good, where each member of $F$ is continuous. Define $\lambda=\max t$ satisfying $t\leq f_i(x_i)/w_i$ for each $i\in [n]$, where the maximum is taken over all positive real values $x_1, \ldots, x_n$ with ${\sum}_i x_i =m$. Then $\lambda$ is achieved when for each $i\in [n]$, $\lambda = f_i(x_i)/w_i$, or equivalently $x_i=f_i^{-1}(\lambda w_i)$.\label{indiv}
\end{prop}

\noindent \begin{proof}
Since $\{(x_1, \ldots, x_n):0\leq x_i~\&~{\sum}_i x_i =m\}$ is a compact subset of $\mathbb{R}^n$ then there exists an $n$-tuple $(x_1, \ldots, x_n)$ with ${\sum}_i x_i =m$ such that $\lambda={\min}_i f_i(x_i)/w_i$. If necessary with a suitable relabeling, assume that $\lambda = f_1(x_1)/w_1 \leq \cdots \leq f_n(x_n)/w_n$. If the assertion is not valid then there exists $k\geq 1$ such that $f_1(x_1)/w_1= \cdots = f_k(x_k)/w_k <f_{k+1}(x_{k+1})/w_{k+1}$. Now for a suitable positive number $\theta$ and for a new allocation $x_1+\theta, x_2, \ldots, x_k, x_{k+1}-\theta, x_{k+2}, \ldots, x_n$, we have $f_2(x_2)/w_2 = f_3(x_3)/w_3 = \cdots = f_k(x_k)/w_k \leq f_1(x_1+\theta)/w_1 \leq f_{k+1}(x_{k+1}-\theta)/w_{k+1} \leq \cdots$. By repeating the above technique we arrive at an allocation $z_1, \ldots, z_n$ for which ${\min}_i f_i(z_i)/w_i > {\min}_i f_i(x_i)/w_i=\lambda$. This contradicts the property of $\lambda$.
\end{proof}

\begin{thm}
Let $m$ be a positive integer and $S$ be a $(1;F,W)$-scenario, such that $F=(f_1, \ldots, f_n)$ and for each $i$, $f_i$ is surjective onto the interval $\{y: 0\leq y \leq f_i(m)\}$. Then there exists an algorithm of time complexity ${\mathcal{O}}(mn\log n)$ which provides a weighted maximin allocation for $S$ with exactly $t$ items for each $t\leq m$.\label{R1item}
\end{thm}

\noindent \begin{proof}
The maximin welfare for the scenario is the maximum $\lambda$ such that there exists an allocation vector $(x_1, \ldots, x_n)$ with ${\sum}_i x_i =m$ such that $\lambda \leq f_i(x_i)/w_i$, for each $i$. If all variables $x_i$ are allowed to take non-negative real numbers then by Proposition \ref{indiv}, $\lambda$ is achieved when $x_i=f_i^{-1}(\lambda w_i)$. But since the items are indivisible then $x_i$'s should be non-negative integers. We begin with a claim.

\noindent {\bf Claim.} Let $\lambda$ be the maximum weighted Rawlsian welfare for $S$ and $\mu$ be a non-negative number such that ${\sum}_i \lceil f_i^{-1}(\mu w_i)\rceil =m$. Then $\lambda=\min_i f_i\big(\lceil f_i^{-1}(\mu w_i)\rceil\big)/w_i$.

\noindent Let $x_i=\lceil f_i^{-1}(\mu w_i)\rceil$, for each $i\in [n]$. Note that $x_i$ is well-defined since $\lceil f_i^{-1}(\mu w_i)\rceil \leq m$ and hence $0\leq \mu w_i\leq f_i(m)$. Assume on the contrary that for some $j\in [n]$, $\lambda > f_j\big(\lceil f_j^{-1}(\mu w_j)\rceil\big)/w_j \geq \mu$. There exists an allocation vector $(y_1, \ldots, y_n)$ such that for some $k$, $f_k(y_k)/w_k = \lambda > \mu$. It follows that for each $i$, $f_i(y_i)/w_i > \mu$ or $y_i > f_i^{-1}(\mu w_i)$. Since each $y_i$ is integer then $y_i\geq \lceil f_i^{-1}(\mu w_i)\rceil=x_i$. But the total sums in both sides are equal to $m$. Hence, $y_i=x_i$, for each $i$ and then $\lambda= {\min}_i f_i(y_i)/w_i={\min}_i f_i(x_i)/w_i$. This contradiction proves the claim.

\noindent Now, in the equation ${\sum}_i \lceil f_i^{-1}(\mu w_i)\rceil =m$, consider $\mu$ as an indeterminate and solve the equation for $\mu$. In order to obtain an efficient procedure for solving the equation we form the following non-decreasing sequence $S_i$ for each $i\in [n]$. $$S_i=\big(\frac{f_i(0)}{w_i}, \frac{f_i(1)}{w_i}, \ldots, \frac{f_i(m)}{w_i}\big).$$
\noindent We use a counter which counts the elements in each $S_i$ by passing the terms from left to right. Fix an arbitrary $i$ and assume that the counter counts exactly $p$ terms in $S_i$. Let $t$ be such that $f_i(p-1)/w_i < t < f_i(p)/w_i$. Then $p-1 < f_i^{-1}(t w_i) < p$. This means that the value of $i$-th summand in ${\sum}_i \lceil f_i^{-1}(t w_i)\rceil$ is $p$. We obtain a general result concerning the main sum as follows. Merge the sequences $S_1, \ldots, S_n$ and obtain a sorted non-decreasing sequence $T$ possibly with repeated terms (i.e. cases such as $f_i(p)/w_i=f_j(q)/w_j$ for some $i,j,p,q$). The counter counts the elements of $T$ from the smallest element to the higher terms such that a repeated term with multiplicity say $\ell$ is counted $\ell$ times. Now, assume that in some step, the counter counts say $q$ terms in $T$. Let $q={\sum}_i q_i$, where the counter counts $q_i$ terms in $S_i$. Apply the latter argument for each $(i,q_i)$ and obtain $t$ between the last counted term and the first non-counted term in $T$ such that ${\sum}_i \lceil f_i^{-1}(t w_i)\rceil=q$.

\noindent Converse of the above fact is also valid. First note that if there exists $t$ satisfying ${\sum}_i \lceil f_i^{-1}(t w_i)\rceil=q$ then there exists $t'\leq t$ such that ${\sum}_i \lceil f_i^{-1}(t' w_i)\rceil=q$ and $f_i^{-1}(t' w_i)$ is not integer. The reason is that the multiset $\{\{f_i(j)/w_i: 1\leq i\leq n, 0\leq j\leq m\}\}$ is finite. Let $\lceil f_i^{-1}(t' w_i)\rceil=q_i$. Then $q_i-1< f_i^{-1}(t w_i) < q_i$ and then $f_i(q_i-1)/w_i < t < f_i(q_i)/w_i$, for each $i\in [n]$. This means that at some step, the counter counts exactly $q$ terms in $T$. Note that because of counting the multiplicities, the set consisting of the numbers counted by the counter is not necessarily a continuous subset of integers. There are then two possibilities for the values of the counter.

\noindent Case 1. In some step of the counting, the counter counts $m$ terms in sequence $T$.

\noindent In this case recall $t$ from the previous paragraph satisfying ${\sum}_i \lceil f_i^{-1}(t w_i)\rceil =m$. Let $\lceil f_i^{-1}(t w_i)\rceil=x_i$. By the claim $(x_1, \ldots, x_n)$ is a maximin allocation for the scenario. Suppose that the counter counted $c_i$ elements in $S_i$. Then ${\sum}_i c_i=m$. Since $c_i$ elements have been passed by the counter then $f_i(c_i-1)/w_i < t< f_i(c_i)/w_i$. This implies that $x_i=\lceil f_i^{-1}(t w_i)\rceil = c_i$. In fact we have proved that $f_i(x_i-1)/w_i < t< f_i(x_i)/w_i$, for each $i$. It follows that the smallest term in $T$ not counted by the counter is ${\min}_i f_i(x_i)/w_i$. But the latter term is the maximin welfare of the scenario.

\noindent Case 2. $m$ does not belong to the set of numbers counted by the counter.

\noindent In this case let $m'\leq m$ be the largest integer counted by the counter. Let $\lambda'$ (resp. $\lambda$) be the maximin welfare corresponding to $m'$ (resp. $m$) items. We prove that $\lambda=\lambda'$. Otherwise, $\lambda > \lambda'$. By Case 1, there exists $t$ such that
$m'={\sum}_i \lceil f_i^{-1}(t w_i)\rceil$ and $\lambda'=f_j\big(\lceil f_j^{-1}(t w_j)\rceil\big)/w_j\geq t$. Hence, $m'\leq {\sum}_i \lceil f_i^{-1}(\lambda' w_i)\rceil < {\sum}_i \lceil f_i^{-1}(\lambda w_i)\rceil$. On the other hand, let $(x_1, \ldots, x_n)$ be a maximin allocation of $m$ items. Then $\lambda \leq f_i(x_i)/w_i$ and then $\lceil f_i^{-1}(\lambda w_i)\rceil \leq x_i$, for each $i\in [n]$. It follows that ${\sum}_i \lceil f_i^{-1}(\lambda w_i)\rceil \leq m$. But Case 2 implies $m\not= {\sum}_i \lceil f_i^{-1}(\lambda w_i)\rceil$. It follows that $m''={\sum}_i \lceil f_i^{-1}(\lambda w_i)\rceil$ satisfies Case 1 and $m'<m''<m$. This contradicts the maximality of $m'$. Then the maximin welfare with $m$ items equals to the maximin welfare with $m'$ items. Using Case 1, consider a maximin allocation with $m'$ items. It follows that there are more than $m-m'$ agents such as $A_{i_1}, \ldots, A_{i_p}$, $p>m-m'$ such that $f_{i_j}(x_{i_j})/w_{i_j}$ is the maximin welfare in $S$ with $m$ (and also $m'$) items. Offer exactly one item (from the remaining $m-m'$ items) to arbitrary $m-m'$ agents from those agents.

\noindent We discuss the complexity of the algorithm. Assume that for a single non-negative integer $p$, $f(p)$ is computed using ${\mathcal{O}}(1)$ arithmetic operations. Since $f$ is increasing then it follows that the sorted sequence $S_i$ is determined with time complexity ${\mathcal{O}}(m)$. We have $n$ sequences and to sort the whole elements takes ${\mathcal{O}}(mn\log n)$ time steps. The counter counts at most $mn$ times. The previous argument provides a maximin allocation with exactly $t$ items for each $t\leq m$. The overall complexity is ${\mathcal{O}}(mn\log n)$.
\end{proof}

\noindent The following lemma will be used in the next results.

\begin{lemma}
Given a $(1;F,W)$-scenario $S$ under the assumptions of Theorem \ref{R1item}, let $(x_1, \ldots, x_n)$ with ${\sum} x_i=m$ be a maximin allocation for $S$ provided by Theorem \ref{R1item} such that ${\sum}_i \lceil f_i^{-1}(t w_i)\rceil =m$ for some positive $t$. Then for each $i,j$, $f_i(x_i)/w_i > f_j(x_j-1)/w_j$.\label{lem}
\end{lemma}

\noindent \begin{proof}
By the proof of Theorem \ref{R1item}, $x_i=f_i(\lceil f_i^{-1}(t w_i) \rceil)$ for each $i\in [n]$. We prove that for any two functions $f_i$ and $f_j$, $f_i(\lceil f_i^{-1}(t w_i) \rceil)/w_i> f_j(\lceil f_j^{-1}(t w_j) \rceil-1)/w_j$. For this end, assume that $f_j^{-1}(t w_j)=\ell +p$ for some integer $\ell$ and $0< p <1$. It follows that the right side is $f_j(\ell)/w_j$. The left side is at least $t$. It is enough now to prove $t> f_j(\ell)/w_j$. The latter inequality is equivalent to $\ell < f_j^{-1}(t w_j)=\ell +p$, as desired. In case that $f_j^{-1}(t w_j)=\ell$ is integer the intended inequality is reduced to $t > f_j(\ell -1)/w_j$ or $t w_j > f_j(\ell-1)$, which clearly holds.
\end{proof}

\noindent Leximin as a social welfare ordering has distinguished fairness and welfare properties \cite{S}. It was proved that a Leximin allocation is EQX for additive scenarios \cite{FSVX}.

\begin{prop}
With the assumptions of Theorem \ref{R1item}, the algorithm of Theorem \ref{R1item} provides a weighted Leximin allocation for $(1;F,W)$-scenarios.\label{leximin}
\end{prop}

\noindent \begin{proof}
The proof is divided into two cases.

\noindent Case 1. There exists $t$ such that ${\sum}_i \lceil f_i^{-1}(t w_i)\rceil =m$.

\noindent Let $x_i=\lceil f_i^{-1}(t w_i) \rceil$, $i=1, \ldots, n$ and assume without loss of generality that $f_1(x_1)/w_1\leq f_1(x_2)/w_2\leq \cdots \leq f_n(x_n)/w_n$ so that $f_1(x_1)/w_1$ is the weighted maximin welfare in the allocation $B=(x_1, \ldots, x_n)$. By Lemma \ref{lem}, $f_1(x_1)/w_1 > f_i(x_i-1)/w_i$, for each $i$. Let $B'$ be a lexicographically improvement of $B$. Assume that $B'$ offers $y_i$ items to $A_i$, $i=1, \ldots, n$. To get a contradiction we claim that $x_i=y_i$. Otherwise, since ${\sum}_i x_i = {\sum}_i y_i$ then there exists $j$ such that $y_j\leq x_j-1$. We have $f_j(y_j)\leq f_j(x_j-1)$ and hence $f_j(y_j)/w_j<f_1(x_1)/w_1$. It follows that the maximum Rawlsian welfare of $B'$ is less than that of $B$. This contradicts the fact that $B'$ is improvement of $B$. This proves the proposition in this case.

\noindent Case 2. Let $m'$ be the largest integer with $m'< m$ satisfying ${\sum}_i \lceil f_i^{-1}(t w_i)\rceil =m'$ for some $t>0$.

\noindent Let $(x_1, \ldots, x_n)$ be a maximin allocation with $m'$ items such that $f_1(x_1)/w_1$ is the maximin welfare in the allocation. Let also $p$ be the largest integer such that $f_1(x_1)/w_1= f_1(x_2)/w_2= \cdots = f_p(x_p)/w_p$.

\noindent {\bf Claim.} $p>m-m'$.

\noindent Assume on the contrary that, $p\leq m-m'$. The fact that $m'$ is the largest integer with $m'< m$ such that the equation ${\sum}_i \lceil f_i^{-1}(t w_i)\rceil =m'$ has a solution for $t$ implies that no allocation of at most $m$ items has maximin welfare strictly greater than $f_1(x_1)/w_1$. From the other side, $p\leq m-m'$. We have $m-m'$ extra items for allocation. Since $m-m' \geq p$ then we have enough items to offer at least one item to each $A_i$, $i=1, \ldots, p$ from the remaining $m-m'$ items. By this allocation, the utility of each $A_i$ is strictly increased and becomes greater than $f_i(x_i)$. In other words, we obtain an allocation with at most $m$ items such that its minimum relative Rawlsian welfare is greater than $f_1(x_1)/w_1$, a contradiction. This proves the claim.

\noindent It follows from the previous paragraph and by $m-m'<p$ that in order to obtain a Leximin allocation for the scenario we should allocate at most one item to each agent in $\{A_1, \ldots, A_p\}$. We form the vector $\big(f_1(x_1+1)/w_1, f_1(x_2+1)/w_2, \cdots, f_p(x_p+1)/w_p\big)$. Choose the $m-m'$ largest terms in the vector and assume that their indexes are $j_1, \ldots, j_{m-m'}$. Offer exactly one item to each agent $A_{j_k}$, $k=1, \ldots, m-m'$. The bundle of each agent $A_i$, $i\not\in \{j_1, \ldots, j_{m-m'}\}$ is left unchanged. Since we choose the first $m-m'$ largest terms from the latter vector, it is easily observed that the total allocation of $m$ items is a weighted Leximin allocation.
\end{proof}

\noindent For a positive $c$, we say a function $f:\mathbb{R}^{\geq 0}\rightarrow \mathbb{R}^{\geq 0}$ is $c$-multiplicative if for each $x, y$, $f(xy)=f(x)f(y)/c$. Let $f$ be a $c$-multiplicative function with $f(1)\not= 0$. Then $f(1)=c$ and $f(1)=f(x)f(1/x)/c$ or $f(1/x)=c^2/f(x)$. It follows that for each $x$ and $y$, $f(x/y)=cf(x)/f(y)$.

\begin{prop}
Let $S$ be a $(1; F, W)$-scenario, where $F=\{f_1, \ldots, f_n\}$ and $W=\{w_1, \ldots, w_n\}$.

\noindent (i) If $w_i=1$, for each $i$, then the scenario admits an EFX+PO allocation.

\noindent (ii) If $W$ is arbitrary weight vector and each $f_i$ is $c_i$-multiplicative (for some $c_i>0$) with $f_i(1)\not=0$ and $w_i$ belongs to the range of $f_i$ then $S$ admits an EFX allocation.\label{efx}
\end{prop}

\noindent \begin{proof}
\noindent To prove $(i)$, let $y_1, \ldots, y_n$ be an EFX allocation for $S$. Then for an $i$ and each $j$, $f_i(x_i)\geq f_i(x_j-1)$, equivalently $x_i\geq  x_j-1$. Hence the allocation is EFX if and only if $|x_i-x_j|\leq 1$. It is clear that every positive integer $m$ is represented as $m={\sum}_i x_i$ satisfying $|x_i-x_j|\leq 1$, as desired.

\noindent To prove $(ii)$, we need a solution vector $(x_1, \ldots, x_n)$ such that for each $i,j$, $f_i(x_i)/w_i\geq f_i(x_j-1)/w_j$. Define positive real numbers $t_i$ as $t_i=f_i^{-1}(w_i)$. Then the inequality $f_i(x_i)/w_i\geq f_i(x_j-1)/w_j$ is equivalent to $f_i(x_i/t_i)\geq f_i((x_j-1)/t_j)$ or
$$\frac{x_i}{t_i} \geq \frac{x_j-1}{t_j}, \forall i, j \in [n].~~~~~\clubsuit$$
\noindent To explore existence of $(x_1, \ldots, x_n)$ satisfying $\clubsuit$, we first check existence of a positive real number $\lambda$ such that $x_i=\lceil \lambda t_i \rceil$ and ${\sum}_i \lceil \lambda t_i \rceil = m$. If such a $\lambda$ exists then $\lceil \lambda t_i \rceil/t_i \geq \lambda > (\lceil \lambda t_j \rceil -1)/t_j$. Hence, $\clubsuit$ holds in this case.

\noindent Otherwise, let $m' < m$ be a largest integer such that ${\sum}_i \lceil \lambda_0 t_i \rceil = m'$ has a solution for some $\lambda_0>0$. Write $\lceil \lambda_0 t_i \rceil=\lambda_0t_i +f_i$ for some $0\leq f_i<1$. Choose $\epsilon$ such that $\epsilon t_i<f_i$, for each $i$ with $f_i\not= 0$. Let ${\sum}_i \lceil (\lambda_0+\epsilon)t_i \rceil = m''$, where $m''$ is an integer with $m''>m$. Let $j\in [n]$ be an index such that $\lceil (\lambda_0+\epsilon)t_j \rceil \geq \lceil \lambda_0 t_j \rceil +1$. Write $\lceil (\lambda_0+\epsilon)t_j \rceil =\lambda_0t_j+\epsilon t_j+f_j(\epsilon)$, for some $0\leq f_j(\epsilon)<1$. Then
$\lambda_0t_j+\epsilon t_j+f_j(\epsilon)\geq \lambda_0t_j+f_j+1$ and equivalently $f_j+f_j(\epsilon)>\epsilon t_j +f_j(\epsilon)\geq f_j+1$. It follows that $f_j(\epsilon)>1$. This contradiction implies that $\lceil \lambda_0 t_j \rceil =\lambda_0t_j$. Now, there are at least $m-m'$ indexes (by counting multiplicities) such as $i_1, \ldots, i_k$ such that $\lceil (\lambda_0+\epsilon)t_{i_j} \rceil \geq \lceil \lambda_0 t_{i_j} \rceil +1$, for each $j=1, \ldots, k$, where $k=m-m'$. By the latter argument we have $\lceil \lambda_0 t_{i_j} \rceil =\lambda_0t_{i_j}$, for each $j=1, \ldots, k$. We add exactly one item to the bundle of each $A_{i_j}$. For each $i$, $1\leq i\leq n$ and $j$, $1\leq j\leq k$ we have $(\lceil \lambda_0 t_i\rceil)/ t_i \geq \lambda = (\lceil \lambda_0 t_{i_j}\rceil)/ t_{i_j}$. It implies that the allocation is $WEFX$ and the total number of allocated items is $m$.
\end{proof}

\noindent Recall that an allocation $B$ is WMMS fair if for each $i$, $u_{A_i}(B_i)\geq WMMS(A_i)$, where $WMMS(A_i)={\max}_{D} {\min}_j (w_i/w_j)u_{A_i}(D_j)$, where the maximum is over all allocations $D$ of $g_1, \ldots, g_m$ to the agents and $D_j=D^{-1}(A_j)$.

\begin{prop}
Existence of an WMMS allocation for a given $(1;W,F)$scenario $S$ can be decided in a polynomial time.\label{mms}
\end{prop}

\noindent \begin{proof}
The weighted maximin share of an agent $A_i$ (denoted simply by $\mu_i$) is determined as $\mu_i={\max}_{(x_1, \ldots, x_n)} {\min}_j (w_i/w_j)f_i(x_j)$. The task is decomposed into determining $\mu_{ij}={\max}_{(x_1, \ldots, x_n)} (w_i/w_j)f_i(x_j)$ (for each fixed $j\in [n]$) subject to $(w_i/w_j)f_i(x_j)\leq (w_i/w_k)f_i(x_k)$ or $\lceil f_i^{-1}\big((w_k/w_j)f_i(x_j)\big)\rceil \leq x_k$, for each $k\in [n]$. For a value $\ell \in \{0, \ldots, m\}$, we put $x_j=\ell$ and decide if there exist non-negative integers $e_1, \ldots, e_{k-1}, e_{k+1}, \ldots, e_n$ such that $\ell+{\sum}_{k\not=j} \big(\lceil f_i^{-1}\big((w_k/w_j)f_i(x_j)\big)\rceil +e_k\big) =m$. We find maximum $\ell$ such that the latter relation holds. This is computed by $\mathcal{O}(m)$ arithmetic operations. Clearly, $\mu_{ij}=(w_i/w_j)f_i(\ell)$. It follows that $\mu_i$ (resp. all $\mu_i$'s) is computed in time complexity $\mathcal{O}(mn)$ (resp. $\mathcal{O}(mn^2)$). After determining all shares $\mu_i$, the final step is to decide whether there exists an allocation $(y_1, \ldots, y_n)$ such that $f_i(y_i)\geq \mu_i$, for each $i\in [n]$. This is equivalent to $m\geq {\sum}_i \lceil f_i^{-1}(\mu_i)\rceil$.
\end{proof}

\section{Equitablity and compensation via identical coins}

\noindent $(1;F,W)$-scenarios with computable invertible functions admits allocations which are both EQX and Pareto optimal.

\begin{prop}
With the assumptions of Theorem \ref{R1item}, the algorithm of Theorem \ref{R1item} provides a {\rm WEQX+PO} allocation for $(1;F,W)$-scenarios.\label{weqx}
\end{prop}

\noindent \begin{proof}
If there exists $t$ such that ${\sum}_i \lceil f_i^{-1}(t w_i)\rceil =m$ (Case 1 in the proof of Theorem \ref{R1item}), then $f_i(\lceil f_i^{-1}(t w_i) \rceil)/w_i$ is the utility of $A_i$ under the maximin allocation of the theorem. By Lemma \ref{lem}, for each $i,j$, $f_i(\lceil f_i^{-1}(t w_i) \rceil)/w_i> f_j(\lceil f_j^{-1}(t w_j) \rceil-1)/w_j$. It follows that the allocation is WEQX in this case. Otherwise, the allocation provided in Case 2 of the proof of Theorem \ref{R1item} is clearly an WEQX allocation.
\end{proof}

\noindent Let $S$ be a general scenario with agent set $A=\{A_1, \ldots, A_n\}$ and weight vector $W=(w_1, \ldots, w_n)$. Let $B: G\rightarrow A$ be an allocation of the goods over the agents. Each agent $A_i$ receives the bundle $B_i=B^{-1}(A_i)$. Define $b_i =u_i(B_i)$ for each $i\in [n]$. Let $p$ be an index such that $b_p/w_p= {\max}_i (b_i/w_i)$. Then weighted deficit of $B_i$ with respect to $B_p$ is $w_ib_p-w_pb_i\geq 0$. Define the ``total weighted deficit" corresponding to allocation $B$ as ${\sf twd}(B)={\sum}_{i=1}^n (w_ib_p-w_pb_i)$. Or, ${\sf twd}(B)=({\sum}_{i\not= p} w_i)b_p - w_p({\sum}_{i\not= p} b_i)$. Note that ${\sf twd}(B)=0$ if and only if $B$ is a weighted equitable allocation. Given a general scenario $S$ with only two agents with equal weights, it is $\NP$-hard to obtain an allocation which minimizes ${\sf twd}(B)$ over all $B$. To prove this fact let $\{\{u_1, \ldots, u_m\}\}$ be an arbitrary multiset (i.e. repeated elements are allowed). Define a scenario with agents $A_1$ and $A_2$ such that the utility of good $i$ is $u_i$ for both agents. Now, minimizing ${\sf twd}(B)$ is equivalent to divide $\{\{u_1, \ldots, u_m\}\}$ into two sub-multisets $P_1$ (as a bundle of $A_1$) and $P_2$ (as a bundle of $A_2$) such that $|({\sum}_{u_i\in P_1} u_i) - ({\sum}_{u_i\in P_2} u_i)|$ has smallest possible value. The latter problem is optimization version of the famous Partition Problem: given a non-empty finite set $P=\{p_i: i\in [n]\}$ of positive integers such that ${\sum}_i p_i=2T$, can $I$ be partitioned into two disjoint subsets $I_1$ and $I_2$ such that ${\sum}_{i\in I_1} p_i={\sum}_{i\in I_2} p_i$? The Partition Problem is $\NP$-hard \cite{GJ}. For a general scenario $S$, define a new weight dependent parameter as follows, where the minimum is taken over all allocations $B$ for $S$.
$$\Psi(S,W)= \min_B {\sf twd}(B).$$
\noindent By the previous argument to determine $\Psi(S,W)$ is $\NP$-hard but Theorem \ref{polypsi} provides a polynomial time algorithm to determine $\Psi(S,W)$ for $(1;F,W)$-scenarios $S$. In case that $w_1=\cdots=w_n=1$ we write $\Psi(S)$ instead of $\Psi(S,W)$. Note that $\Psi(S,W)=0$ if and only if $S$ admits a WEQ allocation. In order to tackle $\Psi(S,W)$ we define $\Psi_p(S,W)$, where $p\in [n]$ is an arbitrary and fixed index. Note that the following minimization is over all allocations $B$ for the scenario such that $u_i(B_i)/w_i \leq u_p(B_p)/w_p$, where as before $B_i$ denotes the bundle allocated to $A_i$ by $B$.

$$\Psi_p(S,W)= \min_{B: \frac{u_i(B_i)}{w_i} \leq \frac{u_p(B_p)}{w_p}} \sum_i \big[w_i(u_p(B_p))-w_p(u_i(B_i))\big].$$

\begin{prop}
$$\Psi(S,W)= \min_{p\in [n]} \Psi_p(S,W).$$\label{psik}
\end{prop}

\noindent \begin{proof}
It is clear that $\Psi(S,W)\leq {\min}_p \Psi_p(S,W)$. For the converse inequality, assume that $\Psi(S,W)$ occurs with an allocation $B$. Suppose $p$ is such that $u_p(B_p)/w_p = \max_i u_i(B_i)/w_i$. It follows that $\Psi_p(S,W)\leq {\sf twd}(B)= \Psi(S,W)$.
\end{proof}

\noindent Given a general scenario $S$, an EQX allocation $B$ for $S$ does not necessarily imply ${\sf twd}(B)=\Psi(S)$ and vise versa. For this purpose just consider a $2\times 3$ utility table with first row all $1$ and the other all $6$. Then $\Psi(S)$ is achieved by allocating all three items of utility 1 to the agent. Figure \ref{bmax} presents a $(1;F,W)$-scenario on $4$ agents with additive utilities, where all weights are one. Three values $\Psi_1(S), \Psi_2(S), \Psi_3(S)$ are computed in the tables from up to down. Each table provides an allocation for which $\Psi_i(S)$ is occurred.

\begin{figure}[h]
\begin{center}
\begin{tabular}{|c|c|c|c|c|c|c|c|}
  \hline
  $\overrightarrow{A_1}$ & $\fbox{2}$ & $\fbox{2}$ & $2$ & $\fbox{2}$ & $\fbox{2}$ & $2$ & $2$\\[0.64eM]
  \hline
   & $4$ & $4$ & $\fbox{4}$ & $4$ & $4$ & $4$ & $4$ \\[0.65eM]
  \hline
   & $7$ & $7$ & $7$ & $7$ & $7$ & $\fbox{7}$ & $7$\\[0.65eM]
  \hline
  & $7$ & $7$ & $7$ & $7$ & $7$ & $7$ & $\fbox{7}$ \\[0.65eM]
  \hline
  \end{tabular}\\
  \vspace*{1cm}
  \begin{tabular}{|c|c|c|c|c|c|c|c|}
  \hline
   & $2$ & $\fbox{2}$ & $\fbox{2}$ & $2$ & $2$ & $2$ & $\fbox{2}$ \\[0.64eM]
  \hline
  $\overrightarrow{A_2}$  & $\fbox{4}$ & $4$ & $4$ & $\fbox{4}$ & $4$ & $4$ & $4$ \\[0.65eM]
  \hline
  & $7$ & $7$ & $7$ & $7$ & $\fbox{7}$ & $7$ & $7$\\[0.65eM]
  \hline
  & $7$ & $7$ & $7$ & $7$ & $7$ & $\fbox{7}$ & $7$ \\[0.65eM]
  \hline
  \end{tabular}\\
  \vspace*{1cm}
  \begin{tabular}{|c|c|c|c|c|c|c|c|}
  \hline
  & $2$ & $2$ & $2$ & $\fbox{2}$ & $\fbox{2}$ & $\fbox{2}$ & $2$\\[0.64eM]
  \hline
   & $4$ & $4$ & $\fbox{4}$ & $4$ & $4$ & $4$ & $4$ \\[0.65eM]
  \hline
  $\overrightarrow{A_3}$  & $\fbox{7}$ & $7$ & $7$ & $7$ & $7$ & $7$ & $\fbox{7}$\\[0.65eM]
  \hline
  & $7$ & $\fbox{7}$ & $7$ & $7$ & $7$ & $7$ & $7$ \\[0.65eM]
  \hline
  \end{tabular}
\caption{From up down: $\Psi_1=6$, $\Psi_2=4$, $\Psi_3=25$ then $\Psi=4$}\label{bmax}
\end{center}
\end{figure}

\noindent The following theorem provides an algorithm for computing $\Psi_p(S,W)$ in case that input instances $S$ are of type $(1;F,W)$-scenarios, where $W$ is arbitrary weight vector and each utility function in $F$ is concave.

\begin{thm}

\noindent (i) $\Psi_p(S,W)$ is computed by an algorithm of time complexity $\mathcal{O}(nm^2)$, where $S$ is an arbitrary $(1;F,W)$-scenario with $n$ agents and concave utility functions, where $m$ is the total number of items to be allocated.

\noindent (ii) There exists an algorithm of time complexity $\mathcal{O}(n^2m^2)$ which obtains an allocation $B$ such that $\Psi(S,W)={\sf twd}(B)$.

\noindent (iii) Deciding whether a given $(1;F,W)$-scenario $S$ has a WEQ allocation is a polynomial time problem.\label{polypsi}
\end{thm}

\noindent \begin{proof}
To prove $(i)$ first note that $u_p(B_p)$ in the definition of $\Psi_p(S,W)$ is only determined by an integer $t\in [m]$ i.e. $u_p(B_p)=f_p(t)$. Note that $u_p(B_p)\not=0$ otherwise by the definition of $\Psi_p(S,W)$ we have $u_i(B_i)/w_i \leq u_p(0)/w_p=0$. It follows that $\Psi_p(S,W)$ can be expanded as follows. Recall that $B_i=B^{-1}(A_i)$ and $b_i=|B_i|$.
$$\Psi_p(S,W)= \min_{t} \min_{B: \frac{u_i(B_i)}{w_i} \leq \frac{u_p(t)}{w_p}, {\sum}_i b_i=m-t} \big[(\sum_{i\not= p}w_i) f_p(t)- w_p\sum_{i\not= p} u_i(B_i)\big]$$
$$= \min_t \big[(\sum_{i\not= p}w_i) f_p(t) - w_p\max_{B: \frac{u_i(B_i)}{w_i} \leq \frac{u_p(t)}{w_p}, {\sum}_i b_i=m-t} \sum_{i\not= p} u_i(B_i)\big].$$
\noindent For a fixed $t$, consider two terms in the latter expression i.e. $(\sum_{i\not= p}w_i) f_p(t)$ and $$\max_{B: \frac{u_i(B_i)}{w_i} \leq \frac{u_p(t)}{w_p}, {\sum}_i b_i=m-t} \sum_{i\not= p} u_i(B_i).$$
\noindent The first term is computed in time ${\mathcal{O}}(n)$. The second term is an instance of the maximum restricted utilitarian allocation with $\ell_i=(w_i/w_p) f_p(t)$. The algorithm GRUA with the back up of Proposition \ref{r-utilitar} decides by consuming $\mathcal{O}(nm)$ time steps whether the instance admits an $\ell$-consistent allocation or not. We say $t$ is feasible if the corresponding instance admits an $\ell$-consistent allocation. In case that it is feasible, let $Y=Y(t)$ be an allocation of $m-t$ goods to $\{A_1, \ldots, A_n\}\setminus \{A_p\}$ obtained by GRUA such that ${\sum}_i u_i(Y_i)$ is maximum subject to $u_i(Y_i)\leq (w_i/w_p) f_p(t)$ for each $i\not=p$, where $Y_i$ is the bundle offered by $Y$ to $A_i$. Define the following function.

\begin{equation}
\tau(t)=\begin{cases}
({\sum}_{i\not= p}w_i) f_p(t) - w_p{\sum}_{i\not= p} u_i(Y_i) \quad\hspace{0cm} if~t~ is~feasible\\
			
+\infty \quad\hspace{5cm} otherwise.
\end{cases}
\end{equation}

\noindent For a fixed $t$, $\tau(t)$ is computed with time complexity $\mathcal{O}(nm)$. Note also that when $t$ is large enough then $t$ is feasible. For instance for $t=m$ we have to offer nothing to each $A_i$, $i\not=p$. We have $\Psi_p(S,W)={\min}_t \tau(t)$. We conclude that $\Psi_p(S,W)$ is computed by $\mathcal{O}(nm^2)$ comparisons and arithmetic operations. This proves $(i)$.

\noindent Part $(ii)$ is obtained by Proposition \ref{psik}. It implies that $\Psi(S,W)$ is determined by an $\mathcal{O}(n^2m^2)$ algorithm. For $(iii)$ we use the fact that $S$ has a WEQ allocation if and only if $\Psi(S,W)=0$. The proof is obtained by $(ii)$.
\end{proof}

\noindent By the proof of Proposition \ref{polypsi} and algorithm GURA we obtain the following pseudo-code for computing $\Psi_p(S,W)$ for $(1;F,W)$-scenarios $S$.

\noindent {\bf Name:} $\Psi_p(S,W)$
\newline {\bf Input:} A $(1;F,W)$-scenario $S$, where $F=(f_1, \ldots, f_k)$ and $W=(w_1, \ldots, w_n)$
\newline {\bf Output:} An allocation minimizing ${\sf wtd}$ with $A_p$ as pivot, i.e. $\Psi_p(S,W)$

\noindent 1. $\forall i\in [n]$ $x_{i}=\ell_i =\hat{m}=0$\\
\noindent 2. {\bf while} $\hat{m}<m$ {\bf do}\\
\noindent 3. ~~~ $q=\arg \min_{j: x_{j}<m} [f_{p}(x_{j}+1)-f_{p}(x_{j})]$ \\
\noindent 4. ~~~ $x_p \leftarrow x_p+1$\\
\noindent 5. ~~~ $\hat{m} \leftarrow \hat{m}+1$\\
\noindent 6. ~~~ $\ell_i \leftarrow (w_i/w_p)f_p(x_p)$\\
\vspace*{0.3cm}
\noindent 7. ~~~ {\bf while} $\exists i\not=p$: $f_{i}(x_{i}+1) \leq \ell_i$ {\bf do}\\
\noindent 8. ~~~~~~~~ $r=\arg \max_{i: f_i(x_i+1)\leq \ell_i} w_i[f_i(x_i+1)-f_i(x_i)]$\\
\noindent 9. ~~~~~~~~ $x_r \leftarrow x_r+1$\\
\noindent 10. ~~~~~~~ $\hat{m} \leftarrow \hat{m} + 1$\\
\noindent 11. ~~~ {\bf end while}\\
\noindent 12. ~{\bf end while}\\
\noindent 13. {\bf return} $\Psi_p(S,W)= ({\sum}_{i\not=p} w_i) f_p(x_p)-w_p{\sum}_{i\not=p} f_i(x_i)$

\noindent Let $S$ be a general scenario on agent set $\{A_1, \ldots, A_n\}$ having entitlement vector $(w_1, \ldots, w_n)$ and utility function $u_i$ corresponding to $A_i$. Fix an index $k$ corresponding to agent $A_k$ and let $t$ be a positive integer. Define a new scenario $S(w_k,t)$ by adding $t$ new identical goods $c_1, \ldots, c_t$ such that the extended utility function of each $A_i$ (still denoted by $u_i$) over $c_1, \ldots, c_t$ is additive and $u_i(c_j)=1/w_k$, for each $i, j$. In case that $S$ is additive scenario, $S(w_j,t)$ can be represented easily by a table. Figure \ref{augmented} depicts a typical table for a general scenario $S(w_j,t)$. The right part of table shows an allocation of $m$ items to the agents in which $A_i$ has benefited $f_i(x_i)$ utility. The entries specified by boxes in the left part shows the items from $c_1, \ldots, c_t$ which have been offered to each $A_i$. Note that $A_k$ receives no item. In the following proposition we assume that utility functions of the agents take integer values. This is not a serious restriction since assuming that the utilities are rational we can make them integer by replacing each utility $u$ by $\ell u$, where $\ell$ is the least common divisor of all denominations.

\begin{figure}[t]
\setlength{\tabcolsep}{2pt}
\begin{center}
\begin{tabular}{|c|c|c|c|c|c|c|c||c|c|c|c|c|}
  \hline
  $A_1$ & $\frac{1}{w_k}$ & $\frac{1}{w_k}$ & $\frac{1}{w_k}$ & $\frac{1}{w_k}$ & $\textcolor{red}{\fbox{$\frac{1}{w_k}$}}$ & $\textcolor{red}{\fbox{$\frac{1}{w_k}$}}$ & $\textcolor{red}{\fbox{$\frac{1}{w_k}$}}$ & $f_1(x_1)$ & & & & \\[0.95eM]
  \hline
  $A_2$ & $\frac{1}{w_k}$ & $\frac{1}{w_k}$ & $\textcolor{red}{\fbox{$\frac{1}{w_k}$}}$ & $\textcolor{red}{\fbox{$\frac{1}{w_k}$}}$ & $\frac{1}{w_k}$ & $\frac{1}{w_k}$ & $\frac{1}{w_k}$ & & $f_2(x_2)$ & & & \\[0.95eM]
  \hline
  $A_k$ & $\frac{1}{w_k}$ & $\frac{1}{w_k}$ & $\frac{1}{w_k}$ & $\frac{1}{w_k}$ & $\frac{1}{w_k}$ & $\frac{1}{w_k}$ & $\frac{1}{w_k}$ & & & $f_k(x_k)$ & & \\[0.95eM]
  \hline
 $\cdots$ & $\frac{1}{w_k}$ & $\textcolor{red}{\fbox{$\frac{1}{w_k}$}}$ & $\frac{1}{w_k}$ & $\frac{1}{w_k}$ & $\frac{1}{w_k}$ & $\frac{1}{w_k}$ & $\frac{1}{w_k}$ & & & & $~~\cdots~~$ & \\[0.95eM]
  \hline
 $A_n$ & $\textcolor{red}{\fbox{$\frac{1}{w_k}$}}$ & $\frac{1}{w_k}$ & $\frac{1}{w_k}$ & $\frac{1}{w_k}$ & $\frac{1}{w_k}$ & $\frac{1}{w_k}$ & $\frac{1}{w_k}$ & & & & & $f_n(x_n)$\\[0.95eM]
  \hline
  \end{tabular}
\caption{A typical augmented scenario by adding identical goods having value $\frac{1}{w_k}$}\label{augmented}
\end{center}
\end{figure}

\begin{prop}
\noindent Let $S$ be a general scenario on $n$ agents with positive integer entitlements $w_1, \ldots, w_n$. Suppose that utility function of each agent takes integer values.

\noindent (i) Let $k\in [n]$ and $t$ be a smallest integer such that $S(w_k,t)$ admits an WEQ allocation in which $A_k$ receives no items from $\{c_1, \ldots, c_n\}$. Then $t=\Psi_k(S,W)$.

\noindent (ii) Let $t$ be a smallest integer such $S(w_j,t)$ admits an WEQ allocation for some $j$. Then $t=\Psi(S,W)$.\label{aug}
\end{prop}

\noindent \begin{proof}
The goal is to minimize ${\sum}_i y_i$ such that if we offer $y_i$ (new) items to $A_i$ then the resulting allocation is EQ. For this purpose we should have $(u_i+(y_i/w_k))/w_i = u_k/w_k$ or equivalently $y_i=u_kw_i-u_iw_k$, for each $i$. By the assumptions each $y_i$ is integer. Hence, the goal is to minimize ${\sum}_i (u_kw_i-u_iw_k)$. The latter problem is equivalent to determining $\Psi_k(S,W)$. Proof of $(ii)$ is easily obtained by part $(i)$ and Proposition \ref{psik}.
\end{proof}\\

\begin{figure}[t]
\begin{center}
  \begin{tabular}{|c|c|c|c|c||c|c|c|c|c|c|c|}
  \hline
  & $c_1$ & $c_2$ & $c_3$ & $c_4$ & $g_1$ & $g_2$ & $g_3$ & $g_4$ & $g_5$ & $g_6$ & $g_7$ \\[0.5eM]
  \hline
  $A_1$ & $1$ & $1$ & $\textcolor{red}{\fbox{1}}$ & $\textcolor{red}{\fbox{1}}$ & $2$ & $\fbox{2}$ & $\fbox{2}$ & $2$ & $2$ & $2$ & $\fbox{2}$ \\[0.5eM]
  \hline
  $A_2$ & $1$ & $1$ & $1$ & $1$ & $\fbox{4}$ & $4$ & $4$ & $\fbox{4}$ & $4$ & $4$ & $4$ \\[0.5eM]
  \hline
  $A_3$ & $1$ & $\textcolor{red}{\fbox{1}}$ & $1$ & $1$ & $7$ & $7$ & $7$ & $7$ & $\fbox{7}$ & $7$ & $7$\\[0.5eM]
  \hline
  $A_4$ & $\textcolor{red}{\fbox{1}}$ & $1$ & $1$ & $1$ & $7$ & $7$ & $7$ & $7$ & $7$ & $\fbox{7}$ & $7$ \\[0.5eM]
  \hline
  \end{tabular}
\caption{Four coins are enough to obtain WEQ allocation since $\Psi=\Psi_2=4$}\label{coin}
\end{center}
\end{figure}

\noindent Figure \ref{coin} illustrates how Proposition \ref{aug} works. First consider the scenario $S$ in Figure \ref{bmax} and the allocation which provides $\Psi_2(S)=4$. It has $7$ goods with four agents. Utility table of scenario $S(w_2,4)$ is depicted in Figure \ref{coin}. This table provides a coin-offering procedure to obtain an equitable (in general setting a WEQ) allocation by only four coins (in general $\Psi_k(S,W)$ coins) of value $1/w_k$. Note that in the figure, $k=2$, $w_2=1$ and $\Psi_2(S)=4$. The general coin-offering procedure is as follows. Theorem \ref{polypsi} computes in polynomial time an allocation say $B$ which provides utility $u_i$ for each agent $A_i$ and such that for some $k$, ${\sum}_i (u_kw_i-u_iw_k)=\Psi(S,W)$. Now, for each $i\not= k$, Agent $A_k$ (or an external source) offers $u_kw_i-u_iw_k$ identical coins of value $1/w_k$ to $A_i$. The resulting allocation is WEQ by Proposition \ref{aug} and uses the minimum number of coins of value $1/w_k$ for this purpose.


\noindent At the end of the paper, we generalize Theorem \ref{polypsi} for $(k;W,F)$-scenarios. We need a discussion. Let $S$ be a $(k;W,F)$-scenario with $m_j$ identical copies of type $g_j$, $j\in [k]$. Let also $S_j$ be a scenario obtained by restricting $S$ to $m_j$ copies of $g_j$. An allocation for $S$ clearly induces an allocation for each $S_j$. We assume in Proposition \ref{k-polypsi} that for each $i\not=j$, a non-equitable allocation of $m_i$ copies of $g_i$ can be compensated by a suitable allocation of $m_j$ copies of $g_j$. That is in an allocation $X$ of all goods in $S$, allocation of $m_i$ copies of $g_i$ can be balanced by the allocation of $m_j$ copies of $g_j$ (under $X$) and vice versa. Then we can define ${\sf twd}(X)$ as ${\sf twd}(X)=({\sum}_{i\not= p} w_i)X_p- w_p({\sum}_{i\not= p} X_i)$, where $w_i$ is an entitlement of $A_i$ in the whole scenario. Otherwise, $S_1, \ldots, S_k$ are independent and a WEQ allocation for $S$ should be obtained by concatenating WEQ allocations for $S_1, \ldots, S_k$ (each obtained by Theorem \ref{polypsi}).

\begin{prop}
Let $k\geq 2$ be a fixed integer and $S$ be a $(k;W,F)$-scenario on $n$ agents in which an agent $A_i$ has entitlement $w_i$. Let $\frak{m}=(m_1, \ldots, m_k)$ and $m={\sum}_j m_j$ be total number of items to be allocated. Define $\Omega={\max}_i {\sum}_j f_{ij}(m_j)$. Assume that $\Omega$ grows polynomially in $m$ if $m\rightarrow \infty$. Then $\Psi_p(S,W)$ and $\Psi(S,W)$ are computed by polynomial time algorithms in terms of $m+n$. Deciding whether a given $(k;F,W)$-scenario $S$ has a WEQ allocation has a similar time complexity.\label{k-polypsi}
\end{prop}

\noindent \begin{proof}
To prove the first assertion note that $u_p(X_p)$ in the definition of $\Psi_p(S,W)$ is only determined by a $k$-tuple $(i_1, \ldots, i_k)$, where $i_j\in [m_j]$ for each $j$. Recall that $X_i=X^{-1}(A_i)$. We have $u_p(X_p)={\sum}_j f_{pj}(i_j)$. Define $\sigma(i_1, \ldots, i_k)=({\sum}_j f_{pj}(i_j))/w_p$.
It follows that $\Psi_p(S,W)$ can be expanded as follows.
$$\Psi_p(S,W)= \min_{(i_1, \ldots, i_k)} \min_{X: \frac{u_i(X_i)}{w_i} \leq \sigma(i_1, \ldots, i_k), {\sum}_{i\not= p} x_{ij}=m_j-i_j} \big[(\sum_{i\not= p}w_i) {\sum}_j f_{pj}(i_j)- w_p\sum_{i\not= p} u_i(X_i)\big]$$
$$= \min_{(i_1, \ldots, i_k)} \big[(\sum_{i\not= p}w_i) {\sum}_j f_{pj}(i_j) - w_p\max_{X: \frac{u_i(X_i)}{w_i} \leq \sigma(i_1, \ldots, i_k), {\sum}_{i\not= p} x_{ij}=m_j-i_j} \sum_{i\not= p} u_i(X_i)\big].$$
\noindent For a fixed $(i_1, \ldots, i_k)$, $({\sum}_{i\not= p}w_i) {\sum}_j f_{pj}(i_j)$ is computed by  $n+k$ arithmetic operations. For an arbitrary $k$ values $r_1, \ldots, r_k$ with $0\leq r_j \leq \sigma$ such that ${\sum}_j r_j\leq \sigma$, define $\rho_i=(w_1r_i, w_2r_i, \ldots, w_nr_i)$, $i\in [k]$. We have following expansion for the big max term in the latter relation.
$$\max_{X: \frac{u_i(X_i)}{w_i} \leq \sigma(i_1, \ldots, i_k), {\sum}_{i\not= p} x_{ij}=m_j-i_j} \sum_{i\not= p} u_i(X_i)=$$
$$\max_{(r_1, \ldots, r_k): 0\leq r_j \leq \sigma} \big[M_{m_1-i_1}U^{\leq \rho_1}(S_1)+M_{m_2-i_2}U^{\leq \rho_2}(S_2)+\cdots+M_{m_k-i_k}U^{\leq \rho_k}(S_k)\big]~~~~~~\clubsuit$$
\noindent Then a $k$-tuple $(i_1, \ldots, i_k)$ is feasible if and only if there exists $(r_1, \ldots, r_k)$ such that $M_{m_j-i_j}U^{\leq \rho_j}(S_j)\not= -\infty$ for each $j\in [k]$. Total number of $k$-tuples $(i_1, \ldots, i_k)$ is ${\prod}_{j=1}^k (m_j+1)\leq (m/k +1)^k$. Note also that for every $(i_1, \ldots, i_k)$, $\sigma(i_1, \ldots, i_k)\leq ({\sum}_j f_{pj}(m_j))/w_p$. For each $j\in [k]$ and each $r\in \{0, 1, \ldots, (1/w_p){\sum}_j f_{pj}(m_j)\}$, Theorem \ref{polypsi} computes $M_{m_j-i_j}U^{\leq \rho}(S_j)$ with time complexity $nm_j^2$, where $\rho=(w_1r, \ldots, w_nr)$. Hence complexity of the maximum term in $\clubsuit$ is $n{\sum}_j m_j^2[\Omega/w_p +1]$. It follows that $\Psi_p(S,W)$ is computed by time complexity
$$(n+k)(m/k +1)^k+n(m/k +1)^k({\sum}_j m_j^2)[\frac{\Omega}{w_p}+1]^k\leq (m/k +1)^k(n+k+nm^2[\frac{\Omega}{w_p}+1]^k).$$
\noindent The rest of assertions are obtained similar to the proof of Theorem \ref{polypsi}.
\end{proof}

\end{document}